\documentclass[11pt,letterpaper]{article}
\usepackage{verbatim, amsfonts, amsmath, amssymb, amscd, xcolor, amsthm, graphicx, mathrsfs, setspace, mdwlist, float, calc,floatflt,enumitem,comment}
\newtheorem{thm}{Theorem}[section]
\newtheorem{cor}[thm]{Corollary}
\newtheorem{lem}[thm]{Lemma}
\newtheorem{prop}[thm]{Proposition}
\newtheorem{prob}[thm]{Problem}

\theoremstyle{definition}
\newtheorem{defn}[thm]{Definition}

\theoremstyle{remark}
\newtheorem{rem}[thm]{Remark}
\newtheorem{ex}[thm]{Example}

\renewcommand{\d }{{\rm d} }

\renewcommand{\H}{\mathcal{H}}

\newcommand{\HG}{\mathcal{H}(G)}
\newcommand{\GG}{\mathcal{G}(G)}

\newcommand{\acts}{\curvearrowright}

\newcommand{\Ga}{\Gamma}

\newcommand{\GwrZ}{G \op{wr} \mathbb{Z}}

\newcommand{\zwrz}{\mathbb{Z}  \op{wr}  \mathbb{Z}}
\newcommand{\op}{\operatorname}
\newcommand{\z}{\mathbb{Z}}
\newcommand{\zn}{\mathbb{Z}_n}

\begin{document}

\title{Hyperbolic structures on wreath products}
\author{S. H Balasubramanya}
\date{}
\maketitle

\begin{abstract} The study of the poset of hyperbolic structures on a group $G$ was initiated in \cite{ABO}. However, this poset is still very far from being understood and several questions remain unanswered. In this paper, we give a complete description of the poset of hyperbolic structures on the lamplighter groups $\z_n \op{wr} \z$ and obtain some partial results about more general wreath products. As a consequence of this result, we answer two open questions regarding quasi-parabolic structures from \cite{ABO}: we give an example of a group $G$ with an uncountable chain of quasi-parabolic structures and prove that the Lamplighter groups $\z_n \op{wr} \z$ all have finitely many quasi-parabolic structures.  \\

\textbf{MSC Subject Classification:} 20F65, 20F67, 20E08.
\end{abstract}



\section{Introduction}
It is common in geometric group theory to study groups via their actions on a metric space. The easiest way to do so is to convert the group $G$ itself into a metric space by fixing a generating set $X$ and endowing $G$ with the corresponding word metric $\d_X$; the group has a natural cobounded action on the associated Cayley graph $\Ga(G, X)$. However, not all generating sets are equally good for this purpose. For example, the Cayley graph corresponding to $X=G$  has diameter 1 and is therefore quasi-isometric to a point. This Cayley graph retains no information about the inherent structure of the group. On the other hand, the Cayley graph in the case when $X$ is finite retains maximum information about the group. In joint work with Carolyn Abbott and Denis Osin in \cite{ABO}, the author focused on formalizing  this notion. One important notion introduced in that paper is the set of \emph{hyperbolic structures} on $G$, denoted $\H (G)$, which consists of equivalence classes of (possibly infinite) generating sets of $G$ such that the corresponding Cayley graph is hyperbolic; two generating sets of $G$ are equivalent if the corresponding word metrics on $G$ are bi-Lipschitz equivalent. These equivalence classes can be ordered according to the amount of information the associated Cayley graph retains about the group (see Section \ref{prelims} for details); this makes $\HG$ a poset.

One important result about $\HG$ follows from Gromov's classification of groups acting on hyperbolic spaces. Using this classification, it is easy to see that $\H(G)$ can be expressed as the following disjoint union: $$\H(G)=\H_e(G)\sqcup \H_{\ell} (G)\sqcup \H_{qp} (G)\sqcup \H_{gt}(G), $$
where the sets of elliptic, lineal, quasi-parabolic, and general type hyperbolic structures on $G$ are denoted $\H_e(G)$, $\H_{\ell} (G)$, $\H_{qp} (G)$, and $\H_{gt}(G)$ respectively (see Section \ref{prelims} for details). 

Out of these, lineal and general-type actions were well-examined in \cite{ABO}, and several interesting examples and results were obtained. Of special interest were the following results: given any $n \in \mathbb{N}$, there exist (distinct) finitely generated groups $G_n$ and $H_n$ such that $|\H_\ell(G_n)| = n$ and $|\H_{gt}(H_n)| = n$. 

However, the understanding of quasi-parabolic structures is far from being complete. It was shown in \cite[Proposition 4.27]{ABO} that $\zwrz$ has an uncountable antichain of quasi-parabolic structures, but little else is known. The authors of \cite{ABO} consequently posed the following two open questions. 

\begin{prob}\label{q1}
Does there exist a group $G$ such that $\H_{qp}(G)$ is non-empty and finite?
\end{prob}
\begin{prob}\label{q2}  Does there exist a group $G$ such that $\H_{qp} (G)$ contains an uncountable chain?
\end{prob}

The answer to both these questions will be obtained as consequences of the following general theorem for wreath products $\GwrZ$ that we will prove in this paper (Theorem \ref{inintro} below). In what follows, $\mathbb{S}_G$ denotes the poset of proper subgroups of $G$ ordered by inclusion. 

\begin{defn} For any group $G$, we will refer to the following poset as $\mathcal{B}(G)$: Two disjoint copies of $\mathbb{S}_G$ which are incomparable .i.e. every element in one copy of $\mathbb{S}_G$ is incomparable to every element in the other copy of $\mathbb{S}_G$. Both copies of $\mathbb{S}_G$ dominate a common element, which in turn dominates one more element. (see Figure \ref{ggenstr}).
\end{defn}

\begin{thm}\label{inintro} Let $G$ be a group. \begin{itemize}[noitemsep,nolistsep] \item[(1)] Then $\mathcal{B}(G) \subset \H(\GwrZ)$. Specifically, the two copies of $\mathbb{S}_G$ correspond to quasi-parabolic structures on $\GwrZ$. The common element that is dominated by the two copies of $\mathbb{S}_G$ corresponds to a lineal structure on $\GwrZ$. In turn, the lineal structure dominates the trivial structure on $\GwrZ$. 
\item[(2)]If $G =\z_n$, then $\mathcal{B}(G) = \H(\z_n \op{wr} \z)$. 
\end{itemize}
\end{thm} 

\begin{figure}[H]
	\centering
	\def\svgscale{0.6}
	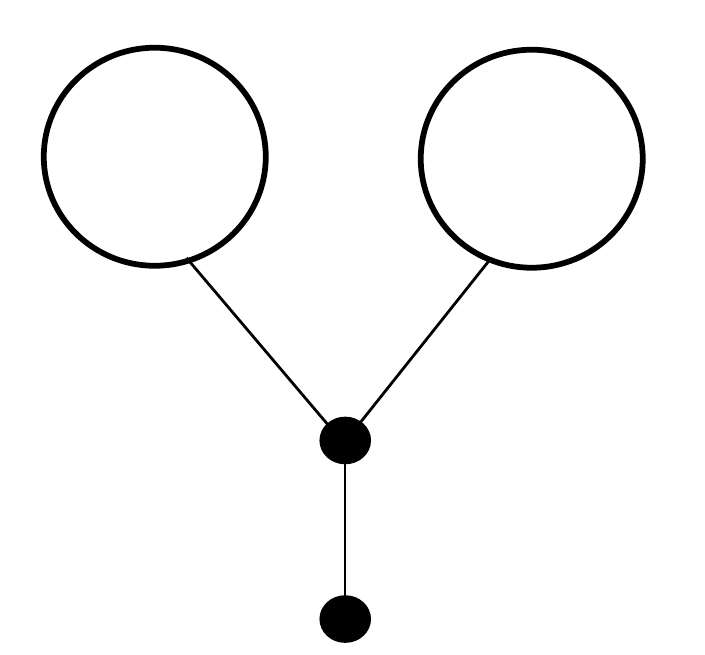
	\caption{$\mathcal{B}(G) \subset \mathcal{H}(G \op{wr} \z)$ }
          \label{ggenstr}
\end{figure} 

The proof of the above theorem is a combination of the description of quasi-parabolic structures obtained (in a different language) in \cite{Amen}, along with elementary (but lengthy) arguments from commutative algebra. Also note that in general, the above equality does not hold. Indeed, the above equality does not even hold for every finite group $G$. The case when $G =\z_2 \times \z_2$ provides a counterexample, as shown in Example \ref{noneqforfin}. 

The answer to Problem \ref{q1} is obtained immediately; the lamplighter groups all have finitely many quasi-parabolic structures. Specifically, the number of quasi-parabolic structures on $\z_n \op{wr} \z$ equals twice the number of proper divisors of $n$. The answer to Problem \ref{q2} is obtained by applying Theorem \ref{inintro} (1) to the group $\mathbb{F}_{2} \op{wr} \z$.
Indeed, it is easy to show that the poset of subsets of $\mathbb{N}$ ordered by inclusion, denoted $\mathbb{P}(\mathbb{N})$, embeds into $\mathbb{S}_{\mathbb{F}_{2}}$. We thus obtain the following corollary.

\begin{cor} \label{uncqp} $\H_{qp}(\mathbb{F}_2 \op{wr} \z)$ contains an isomorphic copy of $\mathbb{P}(\mathbb{N})$. In particular, $\H_{qp}(\mathbb{F}_2 \op{wr} \z)$ contains an uncountable chain. 
\end{cor} 

Although this paper does provide some answers, there are others that remain unanswered. One question relates to the fact that the Lamplighter groups all have an even number of quasi-parabolic structures; these are currently the only examples where groups have a finite number of quasi-parabolic structures. We may thus ask if there exists a group $G$ such that $|\H_{qp}(G)|$ is odd. More generally, we can also whether there exist groups $G_n$ such that $|\H_{qp}(G_n)| = n$, for every $n \geq 1$. One can also ask which conditions the group $G$ would need to satisfy in order to ensure that $\mathcal{B}(G) = \H(\GwrZ)$.

\textbf{Organization of the paper:} The next section will list the required preliminary information. Section \ref{genthm} focuses on proving the first part of Theorem \ref{inintro}, from which Corollary \ref{uncqp} will follow immediately. Section \ref{lamp} discusses the structure of the Lamplighter groups and proves the second part of Theorem \ref{inintro} by giving a complete description of $\H(\z_n \op{wr} \z)$. 

\textbf{Acknowledgments:} I am very grateful to P.E. Caprace and Y. de Cornulier for aiding my understanding of \cite{Amen} by answering my questions. I also sincerely thank Gregory Kelsey for helpful conversations about the Lamplighter groups. As always, my heartfelt thanks to my advisor, Denis Osin, for his continued support and advice.

\section{Preliminaries}\label{prelims}

\paragraph{Comparing generating sets and group actions.} We begin by recalling some standard terminology and definitions. Throughout this paper, all group actions on metric spaces are assumed to be isometric. Given a metric space $S$, we denote by $\d_S$ the distance function on $S$ unless another notation is introduced explicitly. Our standard notation for an action of a group $G$ on a metric space $S$ is $G\curvearrowright S$. Given a point $s\in S$ or a subset $X\subseteq S$ and an element $g\in G$, we denote by $gs$ (respectively, $gX$) the image of $s$ (respectively $X$) under the action of $g$. Given a group $G \curvearrowright S$ and some $s\in S$, $Gs$ denotes the $G$-orbit of $s$.

\begin{defn} An action of a group $G$ on a metric space $S$ is said to be

\begin{enumerate}
\item[(i)] \emph{proper} is for every bounded subset $B\subseteq S$ the set $\{ g\in G\mid gB\cap B\ne \emptyset\}$ is finite;

\item[(ii)] \emph{cobounded} if there exists a bounded subset $B\subseteq S$  such that $S=\displaystyle \bigcup_{g\in G} gB$;

\item[(iii)] \emph{geometric} if it is proper and cobounded.
\end{enumerate}
\end{defn} 

\begin{defn} A map $f\colon R\to S$ between two metric spaces $R$ and $S$ is a \emph{quasi-isometric embedding} if there is a constant $C$ such that for all $x,y\in R$ we have
$$
\frac1C\d_R(x,y)-C\le \d_S(f(x),f(y))\le C\d_R(x,y)+C.
$$
In addition, if $S$ is contained in the $C$--neighborhood of $f(R)$, $f$ is called a \emph{quasi-isometry}.  Two metric spaces $R$ and $S$ are \emph{quasi-isometric} if there is a quasi-isometry $R\to S$. It is well-known that quasi-isometry of metric spaces is an equivalence relation.
\end{defn} 

\begin{defn} If a group $G$ acts on metric spaces $R$ and $S$, a map $f\colon R\to S$ is called \textit{$G$-equivariant} if $f(gr) = gf(r)$ for every $g \in G$ and every $r \in R$. 
\end{defn} 

We now recall the notion of comparing two generating sets of a group, introduced in \cite{ABO}.

\begin{defn}\label{def-GG}\cite[Definition 1.1]{ABO}
Let $X$, $Y$ be two generating sets of a group $G$. We say that $X$ is \emph{dominated} by $Y$, written $X\preceq Y$, if the identity map on $G$ induces a Lipschitz map between metric spaces $(G, \d_Y)\to (G, \d_X)$. This is equivalent to the requirement that $\sup_{y\in Y}|y|_X<\infty$, where $|\cdot|_X=\d_X(1, \cdot)$ denotes the word length with respect to $X$.  It is easy to see that $\preceq$ is a preorder on the set of generating sets of $G$ and therefore it induces an equivalence relation in the standard way:
$$
X\sim Y \;\; \Leftrightarrow \;\; X\preceq Y \; {\rm and}\; Y\preceq X.
$$
This is equivalent to the condition that the Cayley graphs $\Ga(G,X)$ and $\Ga(G, Y)$ are $G$-equivariantly quasi-isometric. We denote by $[X]$ the equivalence class of a generating set $X$, and by $\GG$ the set of all equivalence classes of generating sets of $G$. The preorder $\preceq$ induces an order relation $\preccurlyeq $ on $\GG$ by the rule
$$
[X]\preccurlyeq [Y] \;\; \Leftrightarrow \;\; X\preceq Y.
$$
\end{defn}

For example, all finite generating sets of a finitely generated group are equivalent and the corresponding equivalence class is the largest element of $\GG$. For every group $G$, $[G]$ is the smallest element of $\GG$. Note also that our order is ``inclusion reversing": if $X$ and $Y$ are generating sets of $G$ such that $X\subseteq Y$, then $Y\preceq X$.

To define a hyperbolic structure on a group, we first recall the definition of a hyperbolic space. In this paper we employ the definition of hyperbolicity via the Rips condition. 

\begin{defn} A metric space $S$ is called \emph{$\delta$-hyperbolic} if it is geodesic and for any geodesic triangle $\Delta $ in $S$, each side of $\Delta $ is contained in the union of the closed $\delta$-neighborhoods of the other two sides.
\end{defn} 

\begin{defn}\cite[Definition 1.2]{ABO} 
A \emph{hyperbolic structure} on $G$ is an equivalence class $[X]\in \GG$ such that $\Gamma (G,X)$ is hyperbolic. Since hyperbolicity of a space is a quasi-isometry invariant; this definition is independent of the choice of a particular representative in the equivalence class $[X]$. We denote the set of hyperbolic structures by $\H (G)$ and endow it with the order induced from $\GG$.  \end{defn}

Alternatively, $\HG$ can be described in terms of general cobounded actions of the group $G$ on hyperbolic spaces by utilizing the \textit{Svarc- Milnor map}, which  allows us to interchangeably work with generating sets or cobounded group actions. However, knowledge of this notion is not essential in this paper and so we refer the reader to \cite[Section 3]{ABO} for details. 

\paragraph{General classification of hyperbolic structures.} We now recall some standard facts about groups acting on hyperbolic spaces. For details the reader is referred to \cite{Gro}. Given a hyperbolic space $S$, we denote by $\partial S$ its Gromov boundary. In general, $S$ is not assumed to be proper. Thus the boundary is defined as the set of equivalence classes of sequences convergent at infinity. If $G$ is a group acting on a hyperbolic space $S$, then by $\Lambda (G)$ we denote the set of limit points of $G$ on $\partial S$. That is, $$\Lambda (G)=\partial S\cap \overline{Gs},$$ where $\overline{Gs}$ denotes the closure of a $G$-orbit in $S\cup \partial S$; it is easy to show that this definition is independent of the choice of $s\in S$. The action of $G$ is called \emph{elementary} if $|\Lambda (G)|\le 2$ and \emph{non-elementary} otherwise. The action of $G$ on $S$ naturally extends to a continuous action of $G$ on $\partial S$.

\begin{defn} Given an action of the group $G$ on a hyperbolic space $S$, an element $g\in G$ is called
\begin{enumerate}
\item[(i)] \emph{elliptic} if $\langle g\rangle $ has bounded orbits;
\item[(ii)] \emph{loxodromic} if the map $n \mapsto g^ns ; n \in \z$ is a quasi-isometric embedding for some (equivalently any) $s \in S$
\item[(iii)] \emph{parabolic} otherwise
\end{enumerate} \end{defn}

Every loxodromic element $g\in G$ has exactly $2$ fixed points $g{^{\pm\infty}}$ on $\partial S$, where $g{^{+\infty}}$ (respectively, $g{^{-\infty}}$) is the limit of the sequence $(g{^{n}}s)_{n\in \mathbb N}$ (respectively, $(g{^{-n}}s)_{n\in \mathbb N}$) for any fixed $s\in S$. We clearly have $\Lambda (\langle g\rangle) =\{ g{^{\pm \infty}}\}$.

As a consequence of the Gromov's standard classification of groups acting on hyperbolic spaces and the fact that parabolic actions cannot be cobounded (see \cite[Section 8.2]{Gro}; also \cite{H} for complete proofs in a more general context and some results from \cite[Propositions 3.1 and 3.2]{Amen}), we proved the following classification of hyperbolic structures in \cite{ABO}.
\begin{thm}\label{main00}\cite[Theorem 4.6]{ABO}
For every group $G$, $$\H(G)=\H_e(G)\sqcup \H_{\ell} (G)\sqcup \H_{qp} (G)\sqcup \H_{gt}(G)$$ 
where the sets of elliptic, lineal, quasi-parabolic, and general type hyperbolic structures on $G$ are denoted by $\H_e(G)$, $\H_{\ell} (G)$, $\H_{qp} (G)$, and $\H_{gt}(G)$ respectively. 
\end{thm} 

\paragraph{The Busemann pseudocharacter.} A map $q\colon G\to \mathbb R$ is a \emph{quasi-character} if there exists a constant $D$ such that $$|q(gh)-q(g)-q(h)|\le D$$ for all $g,h\in G$. We say that $q$ has \emph{defect at most $D$}. If, in addition, the restriction of $q$ to every cyclic subgroup of $G$ is a homomorphism, $q$ is called a \emph{pseudocharacter}. Every quasi-character $q$ gives rise to a pseudocharacter $p$ defined by
$$
p(g)=\lim_{n\to \infty} \frac{q(g{^{n}})}n
$$
(the limit always exists); $p$ is called the \emph{homogenization of $q$.} It is straightforward to check that
$$
 |p(g) -q(g)|\le D
$$
for all $g\in G$ if $q$ has defect at most $D$.

Given any action of a group on a hyperbolic space fixing a point on the boundary, one can associate the so-called \emph{Busemann pseudocharacter}. We briefly recall the construction and necessary properties here and refer to \cite[Sec. 7.5.D]{Gro} and \cite[Sec. 4.1]{Man} for more details.

\begin{defn}\label{Bpc}
Let $G$ be a group acting on a hyperbolic space $S$ and fixing a point $\xi\in \partial S$.
Fix any $s\in S$ and let ${\bf x}=(x_i)$ be any sequence of points of $S$ converging to $\xi$. Then  the function $q_{\bf x}\colon G\to \mathbb R$ defined by
$$
q_{\bf x}(g)=\limsup\limits_{n\to \infty}(\d_S (gs, x_n)-\d_S(s, x_n))
$$
is a quasi-character. Its homogenization $p_{\bf x}$ is called the \emph{Busemann pseudocharacter}. It is known that this definition is independent of the choice of $\bf x$ (see \cite[Lemma 4.6]{Man}) and thus we can drop the subscript in $p_{\bf x}$. It is straightforward to verify that $g\in G$ acts loxodromically on $S$ if and only if $p(g)\ne 0$; in particular, $p$ is non-zero whenever $G \curvearrowright S$ is quasi-parabolic.
\end{defn}

\begin{defn} A quasi-parabolic action of a group on a hyperbolic space is called \textit{regular} (or \textit{regular focal}) if the associated Busemann pseudocharacter is a homomorphism. 
\end{defn}

\section{Quasi-parabolic structures on $G \op{wr} \mathbb{Z}$}\label{genthm} 

To prove the first part of Theorem \ref{inintro}, we will first introduce some important notions related to wreath products. The following visualization of some elements from $\GwrZ$ will be very useful in making several arguments in this section: Let $A = \displaystyle \bigoplus_{\mathbb{Z}} G$ denote the base of the wreath product. We naturally think of elements of $A$ as functions from $\mathbb{Z}$ to $G$ with finite support. i.e. only finitely many values in the image can be distinct from $e \in G$. This is a group under co-ordinate wise group operation, which we denote by $+$.

The following notions are taken from \cite{Amen}. These will serve as important tools in proving Theorem \ref{inintro}. 

\begin{defn}\cite[Section 4]{Amen}\label{genconfine} Let $H$ be a group and $Q$ be a subset of $H$, and let $\alpha$ be an automorphism of $H$. We say that the action of $\alpha$ is \textit{(strictly) confining $H$ into $Q$} if it satisfies the following conditions$\colon$ 
\begin{itemize}
\item[(a)] $\alpha(Q)$ is (strictly) contained in $Q$. 
\item[(b)] $H = \displaystyle \cup_{n \geq 0} \hspace{5pt} \alpha^{-n}(Q)$
\item[(c)] $\alpha^{n_0}(Q.Q) \subset Q$ for some non-negative integer $n_0$. 
\end{itemize}
\end{defn}

\begin{prop}\cite[Proposition 4.6]{Amen}\label{niceprop} Let $A$ be a group and let $\alpha$ be an automorphism of $A$ which confines $A$
into some subset $Q \subset A$. Let $S = \{Q, \alpha^{\pm 1} \}$. Then the group $G = A \rtimes \langle \alpha \rangle$  is Gromov hyperbolic with respect to the left-invariant word metric associated to the generating set $S$. If the inclusion $\alpha(Q) \subset Q$ is strict, then it is regular focal.
\end{prop} 

Although the conclusion of the proposition in \cite{Amen} does not state that we get a regular action, it is not too hard to see that this is indeed the case. We include a brief outline (provided by Y. de Cornulier) here: It is clear that $\alpha$ is a loxodromic element with respect to the action of $G$ of $\Ga(G, S)$. If $h$ is the horokernel associated to the sequence $(\alpha^n)$, then $h(1, wa^k) - k$ can be bounded independently of $w \in A, k \in \z$. From this it follows that the associated Buseman pseudocharacater is equal to the map $w\alpha^k \mapsto k$. 

\begin{rem}\label{wegetlin} If the action is confining but not strictly confining .i.e. $\alpha(Q) =Q$, then the above theorem still holds with the difference that $(G, d_S)$ is quasi-isometric to a line; the action is thus lineal. Indeed, if $\alpha(Q) =Q$, then $A =Q$ is bounded. The above theorem implies that when the action is strictly confining, then  $[S] \in \H_{qp}(G)$. 
\end{rem}

To prove Theorem \ref{inintro} (1), we will describe subsets $Q \subset A$ associated to each proper subgroup $H < G$, such that the action of either $t$ or $t^{-1}$ is strictly confining $A$ into $Q$. Note that the action is by conjugation and that we read our elements from left to right. Our standard notation will be that $t.b = tbt^{-1}; t^{-1}.b = t^{-1}bt$. We now restate the conditions of Definition \ref{genconfine} for the specific case of the wreath product. For this, we will need the following definition. Recall that $A$ is the set of functions from $\z$ to $G$ with finite support. 

\begin{defn}[Multiplying a function by $t$ or $t^{-1}$]
Let $\mathbb{Z} = \langle t \rangle$ and $f \in A$ be a function. Then $tf \in A$ is the function given by $$tf(i+1) = f(i), \hspace{5pt} \forall i \in \z.$$ The function $t^{-1}f$ is defined similarly. Observe that since $f$ has finite support, so do $tf$ and $t^{-1}f$. Indeed, $tf$  and $t^{-1}f$ can be thought of as the ``right shift'' and ``left shift'' of $f$ respectively. By convention, multiplication by $t$ or $t^{-1}$ is commutative. 
\end{defn} 

\begin{defn}\label{gennewconfine} Let $\z = \langle t \rangle$ and $Q \subset A$. The action of $t$ is (strictly) confining $A$ into $Q$ if 
\begin{itemize}
\item[(a)] $tQ$ is (strictly) contained in $ Q$  
\item[(b)] For every $ f \in A$, there exists $n \geq 0$ such that $t^{n}f \in Q$
\item[(c)] There exists a constant integer $n_0 \geq 0$ such that $t^{n_0}(Q + Q) \subset Q$.
\end{itemize} 
\end{defn}

We will now prove the first part of Theorem \ref{inintro} in a series of smaller results. We start by showing how to create two quasi-parabolic structures given a proper subgroup of $G$. Recall that $\mathbb{S}_G$ is the poset of proper subgroups of $G$, ordered by inclusion. 

\begin{lem}\label{1diff} Let $H \in \mathbb{S}_G$. Then the action of $t$ is strictly confining $A$ into the set $Q_H \subset A$ given by $$Q_H = \{f \in A \mid f(-i) \in H \hspace{5pt} \forall \hspace{2pt} i > 0\}.$$
\end{lem}

\begin{proof} We prove the lemma by verifying the conditions of Definition \ref{gennewconfine}. Let $f \in Q_H$. Then $tf$ satisfies $tf(-i) = f(- i - 1) \in H $ for all $i \geq 1$. Thus $tf \in Q_H$. To see that the  containment of $tQ_H$ into $Q_H$ is strict, observe that since $H$ is a proper subgroup of $G$,  there exists $g \in G \backslash H$. Define a function $p \in A$ by $p(0) = g$ and identity everywhere else. Then $p \notin tQ_H$. Indeed, if $p \in tQ_H$, then we must have a function $p' \in Q_H$ such that $p'(-1) \notin H$, which is  impossible. This proves condition (a) of Definition \ref{gennewconfine}. 

Given any $f \in A$, $f$ has finite support. Thus, there exists the smallest negative integer $-n$ in the support of $f$. But then $t^nf  \in Q_H$ since $t^nf(-i) = e$ for all $i \geq 1$. Thus condition (b) of Definition \ref{gennewconfine} holds. To see that condition (c) holds, observe that $Q_H + Q_H \subset Q_H$ since $H$ is a subgroup and the functions are combined component-wise. 
\end{proof} 

Similar to the above result, we can prove the following.

\begin{lem}\label{2diff} Let $H \in \mathbb{S}_G$. Then the action of $t^{-1}$ is strictly confining $A$ into the set $Q'_H \subset A$ given by $$Q'_H = \{f \in A \mid f(i) \in H \hspace{5pt} \forall \hspace{2pt} i > 0\}.$$ 
\end{lem}

Note that $Q_H$ and $Q'_H$ are non-empty since the function that takes identity everywhere is in both sets. Thus, each $H \in \mathbb{S}_G$ defines two quasi-parabolic structures on $\GwrZ$ given by the equivalence classes of the generating sets $S_H = \{Q_H, t^{\pm 1}\}$ and $S'_H = \{ Q'_H, t^{\pm 1}\}$. We will fix this notation for the rest of the paper. 

\begin{lem}\label{wdor} The maps $\psi, \psi' \colon \mathbb{S}_G \rightarrow \H_{qp}(G)$ given by $H \mapsto [S_H]$ and $H \mapsto [S'_H]$ respectively are well-defined and order reversing. 
\end{lem}

\begin{proof} We will prove the lemma for the map $\psi$; the arguments for the map $\psi'$ are symmetric. Clearly $\psi$ is a well defined map. Now let $H, K \in \mathbb{S}_G$ such that $H \leq K$. Let $f \in Q_H$. Then it is easy to see that $f \in Q_K$ also. Thus $\op{sup}_{x \in S_H} |x|_{S_K} = 1$, and so $[S_K] \preccurlyeq [S_H]$. Thus $\psi$ is also order reversing. 
\end{proof}

\begin{lem}\label{genembedding} The maps $\psi, \psi'$ defined in Lemma \ref{wdor} are injective.
\end{lem}

\begin{proof} We will prove the lemma for the map $\psi$; the arguments for the map $\psi'$ are symmetric. Let $H, K \in \mathbb{S}_G$ such that $H \neq K$. Without loss of generality, we may assume that $H \not \leq K$. Choose an element $h \in H \backslash K$. Define the following functions in $A$ for each $i \geq 1$:
\begin{center} $p_i(-i) = h$ and $p_i$ is identity everywhere else. \end{center} Then it is easy to check that $\op{sup}_i |p_i|_{S_H} =1$ but $\op{sup}_i |p_i|_{S_K} = +\infty$. Thus $[S_H] \neq [S_K]$.\end{proof}

\begin{lem}\label{genstrareinc} Let $H, K \in \mathbb{S}_G$. Then the structures $[S_H]$ and $[S'_K]$ are incomparable. \end{lem}

\begin{proof}  Since $H < G$, there exists an element $g \in G \backslash H$. It is easy to verify that the functions $f_i \in A$ for $i \geq 1$, given by \begin{center} $f_i(-i) = g$ and $f_i$ is identity everywhere else \end{center} satisfy  $\displaystyle \op{sup}_{i} |f_i|_{S'_K}  = 1$, while $\displaystyle \op{sup}_{i} |f_i|_{S_H}  = +\infty$.  

Similarly, we can find an element $g' \in G \backslash K$. Then the functions $p_i, i\geq 1$ given by \begin{center} $p_i(i) = g'$ and $p_i$ is identity everywhere else \end{center} satisfy the condition that $\displaystyle \op{sup}_{i} |p_i|_{S_H} = 1$, but $\displaystyle \op{sup}_{i} |p_i|_{ S'_{K}} = +\infty$. Thus $[S_H]$ and $[S'_K]$ are incomparable.
\end{proof} 

\begin{proof}[Proof of Theorem \ref{inintro} (1)] By Lemmas \ref{wdor} and \ref{genembedding}, the maps $\psi, \psi'$ create two copies of $\mathbb{S}_G$ in $\H_{qp}(\GwrZ)$. That these copies are incomparable follows from Lemma \ref{genstrareinc}.  

Since all of these structures yield regular quasi-parabolic structures, the Busemann pseudocharacter associated to any of these actions is always proportional to the homomorphism which is the standard epimorphism to $\z$. It follows from \cite[Lemma 4.15]{ABO} that all these (pseudo) characters results in the same orientable lineal action on $\GwrZ$. If $[L]$ denotes this lineal structure on $\GwrZ$, $[L] \preccurlyeq [S_H]$ and $[L] \preccurlyeq [S'_H]$ for every $H \in \mathbb{S}_G$ by \cite[Corollary 4.26]{ABO}. Thus the two copies of $\mathbb{S}_G$ dominate a common lineal structure, which in turn obviously dominates the trivial structure. This results in an embedding of the entire poset from Figure \ref{ggenstr} into $\H(\GwrZ)$. 
\end{proof} 

\begin{cor}\label{cor2} There exits a finitely generated group $W$ such that $\H_{qp}(W)$ contains a chain of cardinality continuum. 
\end{cor} 

\begin{proof} Let $\mathbb{F}_{\infty} = \langle a_1, a_2, a_3, a_4, .... \rangle$ be the free group on infinitely many generators.  For each proper subset $I \subset \mathbb{N}$, set $H_I = \langle a_{2i} \mid  i \in I \rangle$. Then $H_I < \mathbb{F}_{\infty}$ and we have the following properties:
\begin{itemize} 
\item[(1)]  If $I \neq J \subset \mathbb{N}$, then $H_I \neq H_J$ (since the groups are free)
\item[(2)] If $I \subset J$, then $H_I < H_J$. 
\end{itemize} 

This creates an embedding of $\mathbb{P}(\mathbb{N})$ into $\mathbb{S}_{\mathbb{F}_{\infty}}$. Set $W = \mathbb{F} _2 \op{wr} \z$, where $\mathbb{F}_2$ is the free group on 2 generators. Since $\mathbb{F}_\infty < \mathbb{F}_2$, we get the following sequence of embeddings  $$\mathbb{P}(\mathbb{N}) \hookrightarrow \mathbb{S}_{\mathbb{F}_{\infty}} \hookrightarrow \mathbb{S}_{\mathbb{F}_2} \hookrightarrow \mathcal{H}_{qp}(W).$$ Since $\mathbb{P}(\mathbb{N})$ contains chains (and antichains) of cardinality continuum, it follows that so does $\H_{qp}(W)$. \end{proof} 

\begin{rem} Let $G$ be a countable group. Let $H < G$ such that $H = \langle X \rangle$ and $X = \{x_1, x_2, x_3, ...\}$ is countable. For each $i \in \z$, create a copy $X_i$ of $X$, such that $X_i = \{ x_{1i}, x_{2i}, x_{3i},...\}$, where $x_{ji}$ is a copy of $x_j$ in the $i$th copy. Let $R$ be the set of relations that $X$ satisfies and let $R_i$ be a copy of $R$ for each $i \leq -1$. 

Further let the elements of $G$ be enumerated as $X \cup G \backslash X$. For each $i \geq 0$, create a copy $G_i = X_i \cup G \backslash X_i$. Let $S$ be the set of relations that $G$ satisfies and containing R and let $S_i$ be a copy of $S$ for each $i \geq 0$. 

Then $\GwrZ = \left\langle t^{\pm 1}, \{X_i\}_{i \leq -1}, \{G_j\}_{j \geq 0} \mid tx_{ji}t^{-1} = x_{j i+1} , R_i , S_j \hspace{2pt} \forall \hspace{2pt} i \leq -1, j \geq 0 \right\rangle$ is an ascending HNN-extension of $B = \left \langle ...X_{-2}, X_{-1}, G_0, G_1, G_2.... \right \rangle$, and yields precisely the structure $[S_H]$ on $\GwrZ$.  Let $T_H$ be the associated Bass-Serre tree. Since the extension is ascending, the group $\GwrZ$ fixes an end of $T_H$, and hence the action of the group on $T_H$ is quasi-parabolic. Using the Svarc-Milnor map, we can conclude that the structure $[S_H]$ corresponds to a  Cayley graph which is a quasi-tree. 

In particular, we can see that $\mathbb{F}_2$ is countable and the subgroups $H_I$ considered in Corollary \ref{cor2} have countable generating sets. Since the antichain and chain of quasi-parabolic structures of cardinality continuum constructed in Corollary \ref{cor2} correspond to the structures $[S_{H_I}]$, we can conclude that they all correspond to quasi-trees.  
\end{rem} 

\section{Quasi-parabolic structures on the Lamplighter groups}\label{lamp}

The goal of this section is to prove Theorem \ref{inintro} (2). Our approach is to first show that all quasi-parabolic structures of the lamplighter groups are regular, and derive a characterization of quasi-parabolic structures on these groups using techniques developed in \cite{Amen}. 

We begin by discussing the structure of the Lamplighter groups. Denoted $L_n, n \geq 2$, the Lamplighter groups are given by the presentation 
\begin{center}
$\langle a, t \mid [a^{t^i}, a^{t^j}] = e \hspace{5pt} \forall i, j \in \mathbb{Z} , \hspace{2pt} a^n =e \rangle$,
\end{center}
where $x^y = yxy^{-1}$. 

Equivalently, this group is the (restricted) wreath product $(\displaystyle \bigoplus_{\mathbb{Z}} \mathbb{Z}_n) \rtimes \mathbb{Z}$. The ``lamplighter" picture of elements of this group is as follows: Take a bi-infinite road of light bulbs placed at integer points. Each bulb has $n$ possible states corresponding
to the elements of $\mathbb{Z}_n$, given by the powers of $a$. There is also a lamplighter, indicating the particular bulb under consideration. The action of the group on this picture is that $t$ (the generator of $\z$) moves the lamplighter one position to the right, and powers of $a$ change the state of the current bulb under consideration. Thus each
element of $L_n$ can be interpreted as a \textit{configuration} of a finite collection of lit lamps in some allowable states, leaving the lamplighter at a fixed integer. The \textit{null configuration}, denoted $\bar{0}$, is the element where all the lamps are off and the lamplighter stands at index $0$. Note that we read the elements of the group from left to right applied to the null configuration. We will fix the above presentation for the remainder of this paper. 

Let $A = \displaystyle \bigoplus_{\mathbb{Z}} \mathbb{Z}_n$. The visualization of $A$ in $L_n$ is that of elements consisting of a finite number of illuminated lamps in some allowable states, while the lamplighter stands at index $0$. In this section, we will use the description of elements of $L_n$ in terms of group rings in order to make several arguments. Let $R = \mathbb{Z}_n [ \mathbb{Z}]$ be the group ring which Laurent polynomials in $t$ over $\z_n$. For $b \in A$, $k \in \{0,1,...n-1\}$ and $m \in \mathbb{Z}$, define $$kt^m.b= (b^k)^{t^m} = t^{m}b^k t^{-m}. \hspace{40pt} (*) $$ 

\begin{lem} $A$ is a free (left) $R$-module, where the module multiplication is defined by extending (*)  canonically to formal sums in $R$. Further, every $b \in A$ represents as $p(t).a$, where $p(t) \in R$. Thus $A$ is isomorphic to $R$ as $R-$modules. 
\end{lem}
\begin{proof} Observe that $A$ is an abelian group under component wise addition. Further, the module multiplication is well defined and sends elements of $A$ to $A$, since the lamplighter remains at index $0$.  Let $x, y \in R$ and $b,c \in A$. 
Clearly, by definition, we have $(x + y).b = (x.b)(y.b)$. Further since shifting and combining configurations of lamps (with lamplighter at zero) is equal to combining the configurations and then shifting, we also have $x.(bc) = (x.b)(x.c)$. Lastly, it is easy to see that the set $\{ a\} \subset A$ is a basis for $A$ as an $R$-module. 

Now let $b \in A$. Then $ b = (t^{i_1}a^{k_1}t^{-i_1})(t^{i_2}a^{k_2}t^{-i_2})...(t^{i_l}a^{k_l}t^{-i_l})$ for some indices $\{i_j , j =1,2,...l\} \subset \z$ and elements $k_j \in \{1,2,...,n-1\}$. Set $p_b(t) = k_1t^{i_1} + k_2t^{i_2} +... k_lt^{i_l}.$ Then $p_b(t) \in R$ and $p_b(t).a = (t^{i_1}a^{k_1}t^{-i_1})(t^{i_2}a^{k_2}t^{-i_2})...(t^{i_l}a^{k_l}t^{-i_l}) = b.$

Since every ring is a module over itself, $R$ is also an $R-$module. We define a map $f \colon A \rightarrow R$ by the rule $f(b)= p_b(t)$. It is easy to verify that this map is a bijective $R-$ module homomorphism.   
\end{proof} 

\begin{lem}\label{repln} Every element $g \in L_n$ represents as $t^m(p(t).a)$, for some $m \in \mathbb{Z}, p(t) \in R$.  
\end{lem}

\begin{proof} Every element $g \in L_n$ can be represented as $g = t^m b$ for some $m \in \z$ and some $b \in A$. The result now follows from Lemma \ref{repb}. 
\end{proof}

\begin{rem}\label{phomo} The Lamplighter groups $L_n$ are amenable. It thus follows from \cite[Corollary 3.9]{Amen} that all quasi-parabolic actions of the Lamplighter groups are regular. Indeed, the Busemann pseudocharacter
associated to any quasi-parabolic action of $L_n$ is the standard epimorphism to $\z$.
	\end{rem}

We will now derive the following characterization of quasi-parabolic structures on $L_n = \z_n \op{wr} \z$, which is similar to \cite[Theorem 4.1]{Amen}. 

\begin{thm}\label{whatiwant} $[X] \in \H_{qp}(L_n)$ if and only if there exists $Q \subset A$ such that the action of $t$ or $t^{-1}$ is strictly confining $Q$ into $A$ and $[X] = [\{Q, t^{\pm 1} \}]$.  
\end{thm} 

In general, the above characterization may not hold for a generic $G = H \op{wr} \z$. Indeed, in view of \cite[Theorem 4.1]{Amen}, a primary obstruction would be if $G/[G,G]$ contained a free abelian group of rank greater than 1. However, this obstruction is not present when $H =\z_n$. To prove the theorem, we first prove the following lemma.

\begin{lem}[cf Proposition 4.5 in \cite{Amen}]\label{lemma} Let $X$ be a generating set of $L_n$ such that $\Ga(L_n,X)$ is hyperbolic and $L_n \acts \Ga(G,X)$ is quasi-parabolic. Let $p$ be the associated Buseman character to this action, and $B(e, r)$ denote the ball of radius $r$ around the identity $e$ in $\Ga(L_n, X)$. Then $ t \notin \ker(p)$ and there exists an $r_0 > 0$ such that for all $r > 0$, there exists an $n_0$ such that for $n \geq n_0$, $t^n(B(e, r) \cap A)t^{-n} \subset B(e, r_0) \cap A$. In particular, $t$ is confining $A$ into $Q = B(e, r_0) \cap A$. 
\end{lem}

\begin{proof} By Remark \ref{phomo}, the action is regular quasi-parabolic and thus $p$ is a character. If $p(t) =0$, then it can be shown that $p(g) =0$ for every $g \in L_n$; which contradicts the assumption that the action is quasi-parabolic. Thus $t \notin \op{ker}(p)$ and so $t$ is a loxodromic element. Consequently, the global fixed point of the boundary is either the attracting or repelling point of $t$. Up to changing $t$ with $t^{-1}$, we may assume without loss of generality that the global fixed point is the repelling point, which we denote $t^{-\infty}$. In particular, connecting consecutive points of the sequence $1,t^{-1}, t^{-2}, t^{-3},...$ by geodesics defines a quasi-geodesic path to $t^{- \infty}$. Therefore, so does the sequence $g, gt^{-1}, gt^{-2}, gt^{-3},...$ for any $g \in L_n$. 

By \cite[Proposition 3.7]{Amen}, there exists a constant $r_0$ (depending only on the hyperbolicity constant of $\Ga(L_n, X)$), such that any two quasi-geodesic rays with the same end point are eventually $r_0$ close to each other. In particular, it follows from the same proposition that if $|p(g)| \leq C$, then $d_X(t^{-n}, gt^{-n}) \leq r_0 +C$ for all $n$ larger than some $n_0$; where $n_0$ only depends on $d_X(1, g)$. 

Now every element of $A$ has finite order; thus $p(b) = 0$ for every $b \in A$. Consequently, $d_X(t^{-n} , bt^{-n}) \leq r_0$ for all $n$ larger than some $n_0$, which depends only on $d_X(1,b)$. The conclusion that $t^n (B(e, r) \cap A)t^{-n} \subset B(e,r_0) \cap A$ for all $n \geq n_0$ obviously follows; where $n_0$ depends on $r$. The assertion that $t$ is confining $A$ into $Q$ also follows from this. \end{proof} 

\begin{proof}[Proof of Theorem \ref{whatiwant}] The converse implication follows from \cite[Proposition 4.26]{Amen}. Indeed, any subset $Q \subset A$ such that the action of $t$ or $t^{-1}$ is strictly confining $A$ into $Q$ defines the (regular) quasi-parabolic structure given by $[\{Q, t^{\pm 1}\}]$. The forward implication follows from Lemma \ref{lemma}. Given $[X] \in \H_{qp}(L_n)$, we consider the action of $L_n$ on $\Ga(L_n,X)$. This gives us a subset $Q = B(e, r_0) \cap A $ of $A$ such that $t$ confines $A$ into $Q$. Set $S = \{Q, t^{\pm 1}\}$. We will show that the identity map $i \colon  (A \rtimes \langle t \rangle, d_S) \rightarrow (A \rtimes \langle t \rangle, d_X)$ is a quasi-isometry. This map is obviously surjective. Since $S$ is bounded in $d_X$, the map $i$ is Lipschitz. It thus suffices to prove that if a set is bounded in $d_X$, then it is bounded in $d_S$. In restriction to $A$, the result follows from Lemma \ref{lemma}. 

Further, since $p$ is a character and $p(t) \neq 0$; we get that $p(bt^n) = p(b) + np(t) = np(t)$. From the definition of the Buseman character, we can see that $q(g) \leq d_X(1,g)$. Since $|p(g) - q(g)| \leq D$, we get that $-d_X(1,g) -D \leq p(g) \leq d_X(1,g) +D.$  Thus $p$ maps bounded sets of $(L_n, X)$ to bounded subsets of $\mathbb{R}$ and it follows that $[X] = [S]$. If $tQ = Q$, then it follows from Remark \ref{wegetlin} that $[S]$ is a lineal structure, which is impossible. Thus $tQ$ is strictly contained in $Q$.
\end{proof}  


Since $A \cong R$ when $G= \mathbb{Z}_n$, we may prove the results in this section in terms of the elements of $R$  (which are Laurent polynomials) instead of the functions from $A$. This simplifies several notions, for example, the multiplication by $t, t^{-1}$ to elements of $R$ is akin to the multiplication of Laurent polynomials. In view of Theorem \ref{whatiwant}, we will show that any subset $Q \subset A$ such that the action of either $t$ or $t^{-1}$ is strictly confining $A$ into $Q$ must correspond to one the structures $[S_H]$ or $[S'_H]$ respectively, for some $H < G$. To show this, we will first prove the following proposition.
 
\begin{prop}\label{confzn} Let $Q \subset R$ and let the action of $t$ be strictly confining $R$ into $Q$. Let $S = \{Q, t^{\pm 1}\}$. Then there exists a subgroup $H < G$ such that $[S] = [S_H]$. \end{prop} 

Indeed, if this proposition is proved, then by symmetric arguments it will follow that if the action of $t^{-1}$ is strictly confining $R$ into $Q$, then there exists a subgroup $H < G$ such that $[S] = [S'_H]$. This will imply that the set of quasi-parabolic structures on $L_n$ correspond precisely to the two copies of $\mathbb{S}_G$ in $\mathcal{B}(G)$. 

For the sake of convenience, we will let $B_+$ and $B_-$ denote the subsets of $R$ that yield the quasi-parabolic structures corresponding to the trivial subgroup $\{e\}$ of $\mathbb{Z}_n$. i.e. $B_+$ denotes the set of all elements from $R$ with only non-negative powers on $t$ and $B_-$ denotes the set of all elements from $R$ with only non-positive powers on $t$. We also denote by $\widetilde{Q}_H$, the following set $$\{p \in A \mid p(-i) \in H \hspace{4pt} \forall \hspace{2pt} i \geq 1; \hspace{2pt} p(j) = e \hspace{4pt} \forall \hspace{2pt} j \geq 0\}.$$ 

The proof of Proposition \ref{confzn} will involve several steps and lengthy arguments. We provide a brief outline here for clarity: Following the notation from the statement of Proposition \ref{confzn}, we will first show that $[S] \preccurlyeq [\{B_+, t^{\pm 1}\}]$. This will allow us to assume, without loss of generality, that $B_+ \subset Q$. The next step will specify how to identify the required subgroup $H$; it will follow from the definition of $H$ that $[S_H] \preccurlyeq [S]$. The penultimate step will specify how to build all the elements of $\widetilde{Q}_H$ in $Q$ by using the conditions of Definition \ref{gennewconfine}. The last step will be to use these facts to show that $[S] \preccurlyeq [S_H]$; the equality $[S] = [S_H]$ will obviously follow. We start by introducing some terminology. 

\begin{defn} The \textit{negative degree} of $p(t) \in R$ is the smallest negative exponent on $t$ that appears in $p(t)$, if it exists. The \textit{leading negative coefficient} is the coefficient of the term with the negative degree. The \textit{degree} is the largest positive exponent on $t$ that appears in $p(t)$, if it exists; and \textit{leading coefficient} is the coefficient of this term. 
\end{defn}
For the sake of simplicity, we adopt the convention that our elements from $R$ are written such that terms from left to right have increasing powers on $t$. Thus the leading negative term appears at the extreme left and the leading term is written on the extreme right. 

\begin{ex} The negative degree of $p(t) = 2t^{-9} + 3t^{-6} + 1 + 2t + t^6 $ is $-9$  and the leading negative coefficient is $2$. The degree is $6$ and leading coefficient is 1.
\end{ex}

The following results will allow us some flexibility in choosing a representative in the equivalence class of $[S]$. This will be very useful in proving Proposition \ref{confzn}. 

\begin{lem}\label{addbdd} Let $Q \subset R$ and assume that the action of $t$ is strictly confining $R$ into  $Q$. Let $f_j, \hspace{5pt} j \geq 1$ be a collection of elements from $B_+ \backslash Q$ that satisfy the following condition: there exists a constant $K \geq 1$ such that $t^Kf_j \in Q$ for all $j \geq 1$. Let $P = \{t^if_j, \mid j \geq 1, i \geq 0 \}$ and $\overline{Q} = Q \cup P$. Then the action of $t$ is also strictly confining $R$ into $\overline{Q}$. Further, if $S = \{Q, t^{\pm 1} \}$ and $\overline{S} = \{\overline{Q}, t^{\pm 1} \}$, then $[S] =[\overline{S}]$. 
\end{lem}  
\begin{proof} We will verify that $\overline{Q}$ satisfies the conditions of Definition \ref{gennewconfine}. Observe that $tP \subset P$, and so $t\overline{Q} = t(Q \cup P) = tQ \cup tP \subset Q \cup P = \overline{Q}$. To show the containment is strict, observe that since the elements $f_j \in B_+$, there exists a $j_0$ such that $f_{j_0}$ has minimum degree. Then $f_{j_0} \in \overline{Q}$ but $f_{j_0} \notin t\overline{Q}$. Indeed, if $f_{j_0} \in t\overline{Q}$, then either $f_{j_0}$ does not have minimum degree or $f_{j_0} \in Q$. Both of these are impossible by the hypothesis. This proves condition (a) of Definition \ref{gennewconfine}.

Condition (b) of Definition \ref{gennewconfine} is obviously satisfied since $Q \subset \overline{Q}$.  It remains to prove condition (c). Let $m_0 = n_0 +K$, where $n_0$ satisfies condition (c) for $Q$, and $K$ is the constant from the statement of the lemma. It is easy to verify that $t^{m_0}( \overline{Q} + \overline{Q}) \in Q \subset \overline{Q}$. Thus the first claim holds.  

Lastly, $\displaystyle \op{sup}_{s \in S} \hspace{1pt} |s|_{\overline{S}} =1$ since $S \subset \overline{S}$. By the hypothesis,  $\displaystyle \op{sup}_{s' \in \overline{S}} \hspace{2pt} |s'|_S \leq \op{max} \{1, K+ 1 \} < \infty$. Thus $[S] = [\overline{S}]$. 
\end{proof} 

\begin{cor}\label{addfin} Let $Q \subset R$ and assume that the action of $t$ is strictly confining $R$ into $Q$. Let $f_j, \hspace{5pt} 1 \leq j \leq m$ be a finite collection of elements from $B_+ \backslash Q$. Let $P = \{t^if_j, \mid 1 \leq j \leq m, i \geq 0 \}$ and $\overline{Q} = Q \cup P$. Then the action of $t$ is also strictly confining $R$ into $\overline{Q}$. Further, if $S = \{Q, t^{\pm 1} \}$ and $\overline{S} = \{\overline{Q}, t^{\pm 1} \}$, then $[S] =[\overline{S}]$. 
\end{cor}

\begin{proof} For each $j$, there exists a $k_j \geq 1$ such that $t^{k_j}f_j \in Q$, by using condition (b) of Definition \ref{gennewconfine}. Take $K = \op{max} \{k_j \mid 1 \leq j \leq m \}$. Then $t^K f_j \in Q$ for all $1 \leq j \leq m$, by condition (a) of Definition \ref{gennewconfine}.  The result now follows from Lemma \ref{addbdd} with this $K$.
\end{proof}

We are now ready to prove the following result, which is the first step in the proof of Proposition \ref{confzn}.

\begin{lem}\label{atleast} Let $Q \subset R$ and let the action of $t$ be confining $R$ into $Q$. Let $S = \{Q, t^{\pm 1} \}$. Then $[S] \preccurlyeq [\{B_+, t^{\pm 1}\}]$. 
\end{lem}

\begin{proof} We prove the result by showing that $Q$ contains all elements from $B_+$. First observe that by Corollary \ref{addfin}, we may assume without loss of generality that  the constant elements $\{0,1,..., n-1\} \subset Q$. Consequently, by using condition (a) of Definition \ref{gennewconfine}, $$\displaystyle \{1,t,t^2,...,t^i,...\} \cup \{2, 2t, 2t^2,...\}\cup...\cup\{ (n-1), (n-1)t, (n-1)t^2,...\} \subset Q.$$ i.e. $Q$ contains all elements from $B_+$ that contain only one term. 

For any $0 \leq j< i$ and any $r,s \in \zn$ , \hspace{2pt} $rt^j + st^i \in Q +Q$. By using condition (c) of Definition \ref{gennewconfine}, we get that $rt^{n_0 + j} + st^{n_0 + i} \in Q$. i.e. $Q$ contains all elements from $B_+$ with $2$ terms and such that the smallest exponent on $t$ is $n_0$. 

Using Lemma \ref{addbdd}, we may further assume that $Q$ contains all elements from $B_+$ with exactly $2$ terms. Indeed,  this is a consequence of the conclusion of the previous paragraph and the observation that every element $p(t) \in B_+ \backslash Q$ with $2$ terms satisfies $t^{n_0}p(t) \in Q$. By iterating the above steps finitely many times, we can conclude that $Q$ contains all elements from $B_+$ with at most $n_0$ terms. 

It now suffices to prove the following claim: For any $n \geq n_0$, $Q$ contains all elements from $B_+$ with $n$ terms such that the smallest exponent on $t$ is bigger than or equal to $n_0$. Indeed, if the claim is proven, then every element $p(t) \in B_+ \backslash Q$ with more than $n_0$ terms will satisfy $t^{n_0}p(t) \in Q$. Since $Q$ contains all elements from $B_+$ with at most $n_0$ terms already, it will follow from Lemma \ref{addbdd} that $Q$ contains all elements from $B_+$. 

To prove the claim, we will use induction on $n \geq n_0$. Observe that the base of the induction already holds. Assume that the claim is true for all integers $n$ such that $n_0 \leq n \leq k$. i.e. $Q$ contains all elements from $B_+$ with at most $k$ terms such that the smallest exponent on $t$ is $n_0$. Let $$p(t) =  r_1t^{n_0 +i_1} + r_2t^{n_0+ i_2} +... r_{n_0} t^{i_{n_0}+n_0}+...+r_{k+1}t^{n_0+ i_{k+1}},$$ where $0 \leq i_1 < i_2 <...i_{n_0} < ...<i_{k+1}$ is a sequence of non-negative integers and $r_1, r_2,..., r_{k+1}$ is a list of coefficients from $\zn$.

Then $r_1t^{i_1} + r_2t^{i_2} +...+ r_{n_0}t^{i_{n_0}} \in Q$ since this element has $n_0$ terms. We claim that $i_{n_0 + 1} \geq n_0$. Indeed if $i_{n_0 + 1} < n_0$, then we cannot choose $n_0$ non-negative integers less than $n_0 - 1$. Thus $r_{n_0 +1}t^{i_{n_0 +1}} + ... + r_{k+1}t^{i_{k+1}}$ is an element such that the smallest exponent on $t$ is bigger than or equal to $n_0$. This element has $k+1 -n_0 \leq k$ terms. If $k+1 -n_0 < n_0$, this element is in $Q$. If $k+1 -n_0 > n_0$, then this element is in $Q$ by the induction hypothesis. 

Thus $$r_1t^{i_1} + r_2t^{i_2} +...+ r_{k+1}t^{i_{k+1}} \in Q + Q.$$ Using condition (c), we get that $p(t) \in Q$, which completes the proof of the proposition.  \end{proof} 

We now describe how to identify the subgroup $H < G$ given $Q \subset R$ such that the action of $t$ is confining $R$ into $Q$. Set \begin{center} $H = \langle$ All coefficients that appear with arbitrarily small negative powers on t in the elements from $Q \rangle.$ 
\end{center} 

\begin{ex} If $G =\z_8$, and $Q$ only contains the elements $4t^{-i} + 2t^{-i +1} + t^{-1}  + b(t)$, where $ i \geq 2, \hspace{2pt} b(t) \in B_+$, then $H = \langle 0, 4, 2 \rangle = \{0,2,4,6\} < \z_8.$ \end{ex} 

We claim that the subgroup $H$ defined above satisfies $[S] = [S_H]$. To prove this claim, we first establish the following inequality.  

\begin{lem}\label{ineqone} $[S_H] \preccurlyeq [S]$.
\end{lem} 

\begin{proof} For every $k \in \z_n$ which does not appear as one of the generators of $H$, there exists $M_k \geq 0$ such that $t^{-M_k}$ is the most negative degree with which $k$ appears as a coefficient. Take $M = \op{max}_{k} \hspace{1pt} \{M_k\}$; this is a maximum over a finite set and hence exists. It is easy to verify that $t^Mq \in Q_H$ for every $q \in Q$. Thus $\op{sup}_{q \in Q} \hspace{1pt} |q|_{S_H} \leq M+ 1$, and so $[S_H] \preccurlyeq [S]$. 
\end{proof} 

In order to prove the reverse inequality, we will specify an algorithm that will allow us to build in $Q$ all the elements of $\widetilde{Q}_H$. We would like to mention that this process will not involve the use of passing to an equivalent generating set by Lemma \ref{addbdd}; instead it will only use the conditions of Definition \ref{gennewconfine} to build new elements in $Q$. We make some observations here that will simplify this process. 
\begin{itemize} 

\item[(1)] If $H$ is trivial, then it follows that the elements in $Q$ have a uniformly bounded negative degree. Indeed, there must exist and $M$ such that $t^MQ \subset B_+$ and so $[\{B_+ , t^{\pm 1}\}] \preccurlyeq [S]$. It then follows from Lemma \ref{atleast} that $[S] = [\{B_+, t^{\pm 1}\} ] = [S_{\{e\}}]$, and so the claim holds in this case. Thus, we may assume henceforth that $H$ is non-trivial. Consequently, we may assume that the elements of $Q$ have unbounded negative degrees. In particular, this means that $Q$ has infinitely many polynomials in addition to $B_+$. 

\item[(2)]  By Lemma \ref{atleast} we may assume without loss of generality that $B_+ \subset Q$. Since every element of $\z_n$ has finite order, every element of $R$ has an additive inverse. By adding appropriate elements from $B_+$ to any element of $Q \backslash B_+$ and using condition (c) of Definition \ref{gennewconfine}, we can reduce to an element of $B_-$. In particular, we can assume that the elements we are working with when building $\widetilde{Q}_H$ in $Q$ are elements from $B_-$. Note that this process does not alter the subgroup $H$ since the coefficients that generate the subgroup occur with arbitrarily small negative powers on $t$. Thus all these coefficients still appear after this step has been executed. 

\item[(3)] For any element of the form $ r_1t^{-j_1} + r_{2}t^{-j_2} + ... +r_{k}t^{-j_k} \in B_-$, we will refer to the following process as \textit{recovering} an element: Using condition (a) of Definition \ref{gennewconfine}, multiply the element by $t^{j_2}$ to obtain $$r_1t^{-j_1 + j_2} + b(t) \in Q,$$ where $b(t) \in B_+$. Since $B_+ \subset Q$, by adding the inverse of $b(t)$ to the above, we get that $r_1t^{-j_1 +j_2} \in Q+ Q$. Using condition (c) of Definition \ref{gennewconfine}, we \textit{recover} $r_1t^{-j_1 +j_2+ n_0} \in Q$. 

\end{itemize} 

We are now ready to describe the required algorithm. Note that we denote the order of an element $g \in \z_n$ by $O(g)$. 

\textbf{Step 1.} Since $\z_n$ is finite and the negative degrees of the elements in $Q$ are unbounded, there must exist a coefficient $g_1 \in \z_n$ and an infinite sequence of elements $\{p_j\}$ of $B_- \subset Q$ such that the first term of each $p_j$ is of the form $g_1t^{-j_1}$ and $-j_1 \rightarrow -\infty$ as $j \rightarrow \infty$. 

If infinitely many of the elements $p_j$ contain only one term, then we immediately get that $\{g_1 t^{-i} \mid i \geq 1 \}$ in $Q$, by using condition (a) of Definition \ref{gennewconfine}. If not, then there exists a coefficient $g_2 \in \z_n$ and an infinite subsequence $\{q_j\} \subset \{p_j\}$, such that each $q_j$ has the form $$g_1 t^{-j_1} + g_2 t^{-j_2} + ...\hspace{50pt}(*)$$ 

Such a $g_2$ exists because $\z_n$ is finite and the collection $\{ p_j\}$ is infinite. Note that $g_1$ may equal $g_2$. 

\textbf{Step 2.} If the difference $j_1 - j_2$ takes arbitrarily large positive values as $j$ varies for the elements $\{q_j\}$, then the values of $-j_1 + j_2 + n_0$ also take arbitrary small negative values. In this case, by using the process  of recovering elements described above and condition (a) of Definition \ref{gennewconfine}, we can conclude that $$\{g_1t^{-i} \mid  i \geq 0\} \in Q.$$

\textbf{Step 3.} If the difference $j_1 -j_2$ is bounded for the elements $\{q_j\}$, then we proceed to the following cases. Observe that in this case, $-j_2  \rightarrow  -\infty$ also.  Let $O(g_1) = n$ and $O(g_2) =m$.

\textbf{Case 3a.} $n=m$. Then $g_1$ and $g_2$ generate the same cyclic subgroup of $\z_n$, so that $(g_1)^l = (g_2)^{-1}$ for some $l \geq 1$. 
Multiply each element $q_j$ (which has the form (*)) by $t^{j_1 -j_2}$ to obtain the following element in $Q$, by condition (a) of Definition \ref{gennewconfine}: $$g_1t^{j_2} + g_2t^{j_1 -2j_2}+... \hspace{50pt} (\#)$$

Adding $(*)$ and $(\#)$, we get the following element in $Q +Q$:$$g_1t^{-t_1} + g_1g_2t^{-j_2}+...$$ 

By using condition (c) of Definition \ref{gennewconfine}, we get that $$g_1t^{-t_1 + n_0} + g_1g_2t^{-j_2+n_0}+... \in Q$$

By repeating the above steps $l-1$ additional times, we eliminate the second term since $(g_1)^lg_2 =e$ and hence we get new elements of the form $$g_1 t^{j'_1} + g'_{j_3}t^{j'_3} + ... \in Q.$$ Observe that since $-j_1$ gets arbitrarily negative, for all sufficiently large $j$, the exponent $j'_1$ is also negative for all sufficiently large $j$. Further the difference between the exponents on $t$ associated to the term with the leading negative co-efficient and the immediate next term has been increased. i.e. $-j'_1 + j'_3 > j_1 -j_2$. Return to Step 1 with the same choice of $g_1$ and, if needed, choose an appropriate $g_2$ from these new elements. 

\textbf{Case 3b.} $n > m$. In this case, we repeatedly use condition (c) $m-1$ times to add each element $q_j$ to itself $m-1$ time. By doing so, we eliminate the term $g_2t^{-j_2}$ and get new elements of the form $$(g_1)^m t^{j'_1}+ g'_{j_3}t^{j'_3} +... \in Q.$$ 

Observe that since $-j_1$ gets arbitrarily negative, for all sufficiently large $j$, the exponent $j'_1$ is also negative for all sufficiently large j. Further the difference between the exponents on $t$ associated to the term with the leading negative co-efficient and the immediate next term has been increased. i.e. $-j'_1 + j'_3 > j_1 - j_2$. Return to Step 1 with the choice $g_1 = (g_1)^m$ and, if needed, choose an appropriate $g_2$ from these new elements. 

\textbf{Case 3c.} $n < m$. In this case, we repeatedly use condition (c) $n-1$ times to add each of these elements $(*)$  to itself $n-1$ times. By doing so, we eliminate the leading term $g_1t^{-j_1}$ and get elements of the form $$(g_2)^n t^{j'_2}+  g'_{j_3}t^{j'_3} +... \in Q.$$ Observe that since $-j_2$ gets arbitrarily negative, for all sufficiently large $j$, the exponent $j'_2$ is also negative for all sufficiently large $j$. Return to Step 1 with the choice $g_1 = (g_2)^n$ and, and, if needed, choose an appropriate $g_2$ from these new elements. 

The above process will eventually recover elements of the form $$\{ht^{-i} \mid i \geq 1\}$$ in Q, for some $h \in H$. This conclusion follows from the following observations: since $\z_n$ is finite, we have finite choices for coefficients. However, since $Q$ contains infinitely many elements from $B_-$ with unbounded negative degrees, there must be a fixed sequence of coefficients that appear infinitely many times in the elements of $Q$ (possibly with a growing number of terms). The above steps applied to this infinite collection of elements will recover elements with the same coefficients in a specific sequence (which are likely to be different from the original sequence of coefficients). 

Since $(g_i)^n = 1$ for all $i$, we cannot have an infinite sequence of coefficients that decrease in order. Further $O(g^k) = \dfrac{O(g)}{\op{gcd(}k, O(g))}$, so that the orders of elements are reducing in Step 3. This means that Case 3c cannot occur infinitely many times in succession, and we eventually have to apply Cases 3a or 3b. But then the difference between the exponents on $t$ associated to the term with the leading negative co-efficient and the immediate next term increase. This allows us to recover elements with the same coefficients, but with larger negative powers on $t$. Consequently, by using condition (a) of Definition \ref{gennewconfine}, this process results in recovering the elements $\{ht^{-i} \mid i \geq 1\} \subset Q$ for some $h \in H$. 

If $O(h) = |H|$, then we are done. Indeed, we can use condition (c) of Definition \ref{gennewconfine} repeatedly to create all elements $\{ h't^{-i} \mid i \geq 1, h' \in H\}$. To then build $\widetilde{Q}_H$ in $Q$, we use induction on the number of terms in elements of $\widetilde{Q}_H$ and condition (c) of Definition \ref{gennewconfine}.

If $O(h) < |H|$, then we use condition (c) of Definition \ref{gennewconfine} to eliminate all  coefficients from $\langle h \rangle$ from the elements we have thus far in $Q$. Further, we can use the same process to create new coefficients in our elements by combining the existing coefficients with powers of $h$. Now go back to Step 1 and repeat the process. This process must in turn also terminate since $|H| < \infty$ and there exists an element $k \in H$ such that $O(k) = |H|$. 
Thus, we get that $\widetilde{Q}_H \subset Q$. 

\begin{lem} $[S] \preccurlyeq [S_H]$. Hence $[S] = [S_H]$. \end{lem} 
\begin{proof} Let $q \in Q_H$. Then canonically, we have that $q = \widetilde{q} +b$, where $\widetilde{q} \in B_-, b \in B_+$. Since $q \in Q_H$, we must have $\widetilde{q} \in \widetilde{Q}_H$. By the arguments above and Lemma \ref{atleast}, $\widetilde{Q}_H, B_+ \subset Q$. So we get that $q \in Q +Q$. By using condition (c) of Definition \ref{gennewconfine}, $t^{n_0}q \in Q$. Hence $\op{sup}_{q \in Q_H} |q|_{S} \leq n_0 + 1 $, which implies that $[S] \preccurlyeq [S_H]$. By Lemma \ref{ineqone}, $[S] = [S_H]$. 
\end{proof}

Finally, observe that if $H=G$, then the structure $[S] = [S_H]$ is lineal. Thus if the action of $t$ is strictly confining $R$ into $Q$, then $H$ must be a proper subgroup of $G$.  This completes the proof of Proposition \ref{confzn}. 

We now turn our attention to the lineal structures on $L_n$. Our goal is to show that there is a unique lineal structure on $L_n$. 

\begin{lem}\label{tislox} Let $[X] \in \H(L_n)$ be any lineal structure. Then $[X] = [A \cup \{t\}]$.
\end{lem}

\begin{proof} The subgroup $A \leq L_n$ is a characteristic subgroup, and therefore commensurated. Further, $A$ contains no loxodromic elements. It follows from \cite[Lemma 4.21]{ABO} that the induced action of $A$ on $\Ga(L_n, X)$ must be elliptic. 

Now every $g \in L_n$ has the form $g =bt^k$ for some $b \in A, k \in \z$. If $t$ is not a loxodromic element, then $t$ must act elliptically (this dichotomy holds for groups acting along quasi-lines; see \cite[Corollary 3.6]{Man2}). But then no element of $L_n$ can act loxodromically, which is a contradiction. Thus $t$ is loxodromic and by \cite[Corollary 6.7]{ah}, we have $[X] \preccurlyeq [A \cup \{t\}]$. It follows easily that $[X] =[A \cup \{t\}]$.  
\end{proof}

\begin{proof}[Proof of Theorem \ref{inintro}(2)]  The Lamplighter groups have no general type actions.  By using the standard Ping -Pong Lemma, one can show that a group which admits a general type action on a hyperbolic space must contain $\mathbb{F}_2$ as a subgroup. However, the Lamplighter groups are solvable and thus cannot contain non-abelian free subgroups. 

By Proposition \ref{confzn}, the structures from $\mathbb{S}_G$ are the only quasi-parabolic structures on the lamplighter groups. As argued in the proof of Theorem \ref{inintro} (1), each of these quasi-parabolic structures dominates a common lineal structure. This lineal structure is unique by Lemma \ref{tislox}. The trivial structure is always unique and dominated by the lineal structure. Thus $\mathcal{B} = \H(L_n)$. 
\end{proof}

One might expect that $\H(\GwrZ) = \mathcal{B}(G)$ whenever $G$ is finite. However, we have a counter-example to this when $G =\z_2 \times \z_2$. The main obstruction in this case is that elements of the same order do not necessarily generate the same subgroup (unlike the case in $\z_n$) and this gives rise to quasi-parabolic subgroups that do not correspond to elements of $\mathbb{S}_G$. This counterexample also implies that the poset of quasi-parabolic structures on $\GwrZ$ can be very complicated even when $G$ is a very simple group to understand. 

\begin{ex}\label{noneqforfin} Let $G = \z_2 \times \z_2 = \{ (0,0) , (0,1), (1,0), (1,1) \}$. For $i \geq 2$, let $$p_i(t) = (0,1)t^{-i} + (1,0)t^{-i+1},$$ be elements in the group ring $R$. Set $B_+$ to be the collection of all elements from $R$ that have terms only with non-negative powers on $t$. Set \begin{center} $Q = \displaystyle \left\{ \sum_{k \hspace{3pt} finite} t^{j_k} p_{i_k}(t) + b(t) \mid k \geq 0, \hspace{2pt}, j_k \geq 0, \hspace{2pt} i_k \geq 2, \hspace{2pt} b(t) \in B_+ \right\}$ \end{center} 

Clearly, $Q \subset R$. It is understood that when $k =0$, the sum is an empty sum and thus we have elements from $B_+$. We first show that the action of $t$ is strictly confining $R$ into $Q$. It is clear that $tQ \subseteq Q$. To see the strict containment, consider the element $(1,0)t^0 = (1,0) \in B_+ \subset Q$. We will show that $(1,0)t^{-1} \notin Q$, which will imply that $(1,0) \notin tQ$ and hence $tQ \subset Q$. Assume, by contradiction, that $(1,0)t^{-1} \in Q$. Then we must have \begin{equation}\label{decomp} (1,0) t^{-1} =  \sum_{k \hspace{3pt} finite} t^{j_k} p_{i_k}(t) + b(t) \end{equation} for some choice of $k, j_k$'s and $p_{i_k}$'s. Observe that $k \neq 0$ since $(1,0)t^{-1} \notin B_+$. Also we must have $b(t) = 0$ since $(1,0)t^{-1} \in B_-$. 

Further $t^{j_k} p_{i_k}(t)  = t^{j_k} ( (0,1)t^{-i_k} + (1,0)t^{-i_k +1}) = (0,1)t^{-i_k + j_k} + (1,0)t^{-i_k + j_k +1}$.  Fix an $l \geq -1$. If for any $k$, $-i_k + j_k = l$, then the same must hold for an even number of $k$'s so that these terms cancel amongst themselves in (\ref{decomp}). Indeed, if this holds for an odd number of $k$'s, then the sum in (\ref{decomp}) will have terms from $B_+$, which is not possible. 

It thus suffices to consider the case when $-i_k + j_k \leq -2$ for all $k$. In this case,  $t^{j_k} p_{i_k}(t)  = p_{i_k - j_k}(t)$ for all $k$. If $i_k -j_k = i_k' -j_k'$ for distinct $k, k'$; then these elements will cancel each other since $O(0,1) = O(1,0) =2$. Thus we may assume that the $i_k -j_k$'s are all distinct. We may also re-arrange the terms so that if $k$ ranges over $\{1,2,...n\}$, then $i_1 -j_1 > i_2 - j_2 >...>i_n - j_n$. This reduces the sum in $(\ref{decomp})$ to $$(0,1)t^{i_1 - j_1} + ... + (1,0)t^{-i_n + j_n +1},$$ where the terms in the middle have co-efficients $(0,1), (1,0)$ or $(1,1)$. But such a sum has at least 2 non-zero terms and thus cannot equal $(1,0)t^{-1}$. 

Since $B_+ \subset R$, condition (b) of Definition \ref{gennewconfine} holds. It remains to show condition (c). To this end, it is easy to verify that $Q +Q \subseteq Q$. Hence, the action of $t$ is strictly confining $R$ into $Q$. 

Now let $X = \{Q, t^{\pm 1} \}$. We will show that $[X] \neq [S_H]$ for any $H < G$, which will prove the counter-example. Observe that $G =\z_2 \times  \z_2$ has the following proper subgroups: $$ \{(0,0)\} ; \{(0,0) , (0,1)\} ; \{(0,0) , (1,0) \} ; \{(0,0)  (1,1) \}$$

\textbf{Case a.} $H = \{(0,0)\}$. Then it is obvious that $|p_i(t)|_{S_H} \rightarrow +\infty$ as $i \rightarrow +\infty$. 

\textbf{Case b.} $H = \{(0,0) , (0,1)\}$. In this case, $Q_H$ cannot contain an element where a negative power of $t$ has the co-efficient $(1,0)$ since $O(0,1) =2$. Thus $|p_i(t)|_{S_H} \rightarrow +\infty$ as $i \rightarrow +\infty$. The cases when $H =\{(0,0) , (1,0) \}$ and $H =  \{(0,0)  (1,1) \}$ have similar arguments since $O(1,0 ) = O(1,1) =2$ also.

However, $|p_i(t)|_{X} = 1$ for every $i$. Thus $[X] \neq [S_H]$ for any $H < G$. 
\end{ex}

\end{document}